\documentclass[12pt]{amsart}

\usepackage{amsmath,amssymb}
\usepackage{accents}
\usepackage{graphicx}
\usepackage{verbatim}
\usepackage{color}
\usepackage{float}
\usepackage{enumerate}

\newlength{\dhatheight}

\newtheorem{theorem}{Theorem}[section]
\newtheorem{conjecture}[theorem]{Conjecture}

\newtheorem{corollary}[theorem]{Corollary}

\newtheorem{lemma}[theorem]{Lemma}

\begin{document}
\title[Unknotting number and connected sums: $4_1$ and $5_1$]{Unknotting number and connected sums: The knots $4_1$ and $5_1$}

\author{Mark Brittenham and Susan Hermiller}
\address{Department of Mathematics\\
        University of Nebraska\\
         Lincoln NE 68588-0130, USA\\
        mbrittenham2@unl.edu}
\address{Department of Mathematics\\
        University of Nebraska\\
         Lincoln NE 68588-0130, USA\\
        hermiller@unl.edu}

\date{January 25, 2026}

\maketitle

\begin{abstract}
We show that the knots $K\in\{4_1,5_1\}$ can be paired with a corresponding
knot $K^\prime$ such that $u(K\#K^\prime)<u(K)+u(K^\prime)$.
As a consequence unknotting number fails to be additive for these knots.
We also provide a candidate knot $K^\prime$ for the knot $3_1$. 
\end{abstract}




\section{Introduction}\label{sec:intro}


For knots $K$ in the 3-sphere $S^3$, the \emph{unknotting number} (or \emph{Gordian number}) 
$u(K)$ of $K$ is one of the most fundamental
measures of the complexity of the knot. It is defined as the 
minimum number of crossing changes, interspersed with isotopy,
required to transform a diagram of $K$ to a diagram of the unknot. 
Its study dates back to the very early years of knot theory;
John Tait referred to it as a knot's `beknottedness' in the 
first papers on the subject. In 1877, Tait \cite{tait77} stated `\emph{There must be some
simple method of determining the amount of beknottedness for any
given knot; but I have not hit upon it.}'. Nearly a 
century and a half later, the
unknotting number remains a challenging invariant to compute, even for some
knots with small crossing number. Many of these challenging examples occur
as connected sums.

Connected sum is a basic operation for creating or decomposing knots and links.
A diagram for the connected sum $K_1\#K_2$ of two knots $K_1$ and $K_2$ can be obtained from 
disjoint diagrams of the two knots by deleting a short arc in each and 
introducing two new arcs to connect the endpoints of the deleted arcs to
create a knot, without introducing additional crossings in the diagram. 
For a very long time it was an open problem whether or not the 
unknotting number was additive under connected sum; that is, whether or
not $u(K_1\#K_2)=u(K_1)+u(K_2)$ for every pair of knots. This is known as 
the Unknotting Additivity Conjecture \cite{kir93},\cite{wendt37}. The authors, 
however, showed that this was not true in \cite{bh25}, by showing that
the knot $K=7_1$ and its mirror image $\overline{K}$, each with unknotting number $3$, satisfy
$u(K\#\overline{K})\leq 5$. This in turn led to the discovery of 
infinitely many counterexamples, using the notion of Gordian adjacency.
A knot $K_1$ is \emph{Gordian adjacent} to $K_2$ if $K_1$ is contained in 
some minimal unknotting sequence for $K_2$.
Any knot that $7_1$ is Gordian adjacent to will also fail additivity, 
when paired with the knot $\overline{7_1}$ in a connected sum. 
The choice of mirroring, that is, whether
we use the knot $K$ or its mirror image $\overline{K}$ in a connected sum, 
must be taken into account in these constructions, since this can have a 
profound effect on the unknotting
number. The main result of \cite{bh25}, that $u(7_1\#\overline{7_1})\leq 5$,
stands in contrast with the fact that $u(7_1\#7_1)=6$.

Any counterexample to the Unknotting Additivity Conjecture built from a 
knot $K$ that $7_1$ is Gordian adjacent to, however, will have 
$u(K)\geq 3$ (and, except for $7_1$ itself, 
$u(K)\geq 4$), as will the corresponding connected summand 
$7_1$ or $\overline{7_1}$. This raises the natural question of whether or not any simpler
knots, with lower unknotting number, might also fail the Unknotting Additivity
Conjecture, meaning that they form part of a connnected sum that
fails additivity of unknotting number. In this paper we settle this 
question, by showing that the knots
$4_1$ (the figure-8 knot), with unknotting 
number one, and the knot $5_1$ (the $(2,5)$ torus knot), with unknotting
number 2, can be made part of a connected sum for which additivity fails. More precisely,

\begin{theorem}\label{thm:four-one}
The connected sum $K_4=4_1\#9_{10}$, of the knots $4_1$ and $9_{10}$,
satisfies $u(4_1)=1$, $u(9_{10})=3$ and $u(4_1\#9_{10})\leq 3$.
\end{theorem}

\begin{theorem}\label{thm:five-one}
The connected sum $K_5=5_1\#8_2$, of the knots $5_1$ and $8_2$, with mirroring chosen so that
$\sigma(K_5)=0$, satisfies $u(5_1)=2$, $u(8_2)=2$ and $u(5_1\#8_2)\leq 3$.
\end{theorem}

\smallskip

The statement of Theorem \ref{thm:four-one} is independent of mirroring, 
since the knot $4_1$ is isotopic to its mirror image.

We show, further, that the knot $3_1$ (the trefoil)
can be paired with the knot $10_6$ to yield a knot with unknotting
number 3. The knot $10_6$ 
has been listed in databases as having unknotting number $3$ since the 1990s. The
earliest instance we found of this is in 
the book {\it A Survey of Knot Theory} \cite{ka96}. However, we have also found
that this was a typo, corrected in notes published on the author's website \cite{kawcorr}; the 
correct statement is that $10_6$ has unknotting number 2 or 3. 
This information has recently been incorporated into the Knotinfo
database \cite{knotinfo}. This makes $10_6$ the
first, in the standard ordering of knots, of the (now) ten knots with up to 10 crossings whose
unknotting numbers remain unknown. If it is in fact the case that
$u(10_6)=3$, this would provide a further example of the failure of the Unknotting Additivity Conjecture.

\begin{theorem}\label{thm:three-one}
The connected sum $K_3=3_1\#10_6$, of the knots $3_1$ and $10_6$ with mirroring chosen so that
$|\sigma(K_3)|=2$, satisfies $u(3_1)=1$, $u(10_6)\in\{2,3\}$, and $u(3_1\#10_6)\leq 3$.
\end{theorem}

\smallskip

Consequently, when it comes to the additivity of unknotting number under connected sum,
there are now examples which 
demonstrate that `$1+3\leq 3$' and `$2+2\leq 3$'.
In particular, there exists a knot ($9_{10}$, in Theorem \ref{thm:four-one}) 
with the property that connected sum with some other non-trivial knot
($4_1$, from Theorem \ref{thm:four-one}) does not raise the unknotting number! 

We propose to introduce the term symbiont to describe this relationship; that is, 
two knots $K$ and $K^\prime$ are \emph{symbionts} if $u(K\#K^\prime)<u(K)+u(K^\prime)$.
Our rationale for this term is that the two knots are 
in essence `helping' one another to unknot; 
since the unlinking number of the disjoint union $K\coprod K^\prime$
is $u(K)+u(K^\prime)$ (a crossing change to unknot one component does not
change, and therefore cannot help, to unknot the other), 
taking a connected sum makes it easier
to unknot the pair than can be done separately. 

Note that symbiosis
must take into account the mirroring of the knot. 
While $7_1$ and its mirror $\overline{7_1}$ are symbionts, for example,
from the authors' previous work \cite{bh25},
$7_1$ and $7_1$ are not. And Theorem \ref{thm:five-one} shows that, if 
representatives of $K=5_1$ and $K^\prime=8_2$ are chosen so that
$\sigma(K)=-\sigma(K^\prime)=4$, then $K$ and $K^\prime$ are symbionts,
but $K$ and $\overline{K^\prime}$ are not 
(since $\sigma(K\# \overline{K^\prime})=8$, and so $u(K\# \overline{K^\prime})\geq 4$, by \cite[Theorem~10.1]{mu65}).

The examples presented in this paper were found using the same 
basic strategy as those originally found in \cite{bh25}. 
The idea is to find a diagram for each connected sum,
together with a crossing change in each diagram, that results in a 
knot whose unknotting number we know is too low for 
the Unknotting Additivity Conjecture. The data of 
unknotting numbers we use comes from the authors'
work \cite{bh21} on the Bernhard-Jablan Conjecture~\cite{bernhard},\cite{jablan98}.
This conjecture asserted that every knot $K$ possesses
a minimum crossing number projection $D$ and a 
crossing change in $D$ to a diagram $D^\prime$ of
a knot $K\prime$ with $u(K^\prime)=u(K)-1$.
This conjecture was disproved by the authors in \cite{bh21}.
In that work the authors introduced the \emph{Bernhard-Jablan 
unknotting number} $u_{BJ}(K)$, which is an upper bound for 
$u(K)$ computed, recursively, by finding the smallest value of $u_{BJ}(K^\prime)$
among all of the knots one crossing change away from 
$K$ in all of its minimial crossing diagrams, and adding $1$.
That is, the quantity $u_{BJ}(K)$ actually has the property
that the conjecture asserted for $u(K)$. It is effectively computable, 
provided one can determine all of the minimum crossing 
diagrams $K$ and, recursively, all of the minimal diagrams of the
knots $K^\prime$ obtained by changing a crossing in one of the
diagrams for $K$. Details can be found in \cite{bh21}. 
The authors carried out this computation for 
all prime knots with 15 or fewer crossings; these values
were used as a stand-in for the unknotting number, in
our searches.

To find the examples presented here, we carried out our searches by repeatedly building
random diagrams of connected sums $K\#K^\prime$, with $K$ equal to one of the knots
$3_1$, $4_1$, and $5_1$, and the other summand $K^\prime$ taken
from knots with known unknotting number $3$ (in the case of $3_1$ and $4_1$) or 
$2$ (in the case of $5_1$), in Snappy/SageMath
\cite{snappy},\cite{sagemath}. We then carried out random crossing changes
in these diagrams, searching for knots $K^{\prime\prime}$ with BJ-unknotting number,
and hance also unknotting number, at most two.
Since unknotting number changes by at most one 
under crossing change, this would force the unknotting number of the connected sum
to be at most $3$, and so $K$ and $K^\prime$ would satisfy 
$u(K\#K^\prime)<u(K)+u(K^\prime)$. Our searches
succeeded in finding the three pairs described in Theorems
\ref{thm:four-one}, \ref{thm:five-one} and \ref{thm:three-one}. 

Diagrams demonstrating our results were in fact found dozens of times for each pair.
In the end, each example that we found had diagrams which established the results and were 
surprisingly small; more specifically, the diagrams we present each have 15 crossings. 
The details of the proofs of 
Theorems \ref{thm:four-one}, \ref{thm:five-one} 
and \ref{thm:three-one} for each pair are given in Section \ref{sec:counter}.

As noted above, after finding the result in Theorem \ref{thm:three-one} we learned, 
as a result of our search of the literature,
that the summand $10_6$ we found for the knot $3_1$ did not actually 
have known unknotting number $3$.
This, and some of its consequences for unknotting number, is discussed 
further in Section \ref{sec:ten-six}. 

Just as with our first example, $7_1\#\overline{7_1}$, from \cite{bh25}, any knot $K$ that
has one of $4_1$ or $5_1$ in a minimal unknotting 
sequence for $K$
also has the corresponding knot ($9_{10}$ or $8_2$, respectively) as a symbiont, 
and so their connected sum can be unknotted more efficiently than the 
Unknotting Additivity Conjecture predicted, as discussed in Section \ref{sec:further_examples}.
Many knots in the standard knot tables have this property;
we list some of them in Section \ref{sec:further_examples}. 
We note that as a consequence, we can improve upon~\cite[Corollary~2.1]{bh25}:

\vspace{0.1in}

\noindent{\bf Corollary~4.2.}
\emph{Every non-trivial torus knot, except possibly for the trefoil knot $T(2,3)$, 
possesses a symbiont.}

\vspace{0.1in}

In Section \ref{sec:further_examples} we also list those knots that have $3_1$ in a minimal unknotting sequence,
in case there may be a proof in the future that $u(10_6)=3$.


\section{Proofs of the Theorems\label{sec:counter}}


The proof of each of the theorems of the introduction consists of
finding a planar diagram for each connected sum to use for the
first crossing change, and demonstrating an unknotting sequence that is
shorter than the sum of the unknotting numbers of the knot summands. 
As mentioned in Section \ref{sec:intro}, each of our connected sums have a
15-crossing diagram which exhibits the first of the needed crossing 
changes. Here we describe each diagram, and the intermediate knots in
their unknotting sequences, using Dowker-Thistlethwaite (DT) codes \cite{dt83}.
Python code, for use in SnapPy, to carry out all of the needed verifications can be 
found in Section \ref{sec:verify}.

\smallskip

\begin{proof}[Proof of Theorem \ref{thm:four-one}]
Starting with the connected
sum $K_4$ of the knots $4_1$ and $9_{10}$, the knot $K_4$ has a 15-crossing 
diagram with a DT code given by

\begin{center}
$[6, -10, 24, \underline{20}, -4, -22, -8, 26, 28, 30, -12, -2, 14, 18, 16]$,\hskip.2in (*)
\end{center}

\noindent as identified by SnapPy. In Figure \ref{fig:connsum41b} we demonstrate
an isotopy from the connected sum decomposition $4_1\#9_{10}$
of the knot to the 15 crossing diagram defined by the DT code (*).

\begin{figure}[h]
\begin{center}
\includegraphics[width=5in]{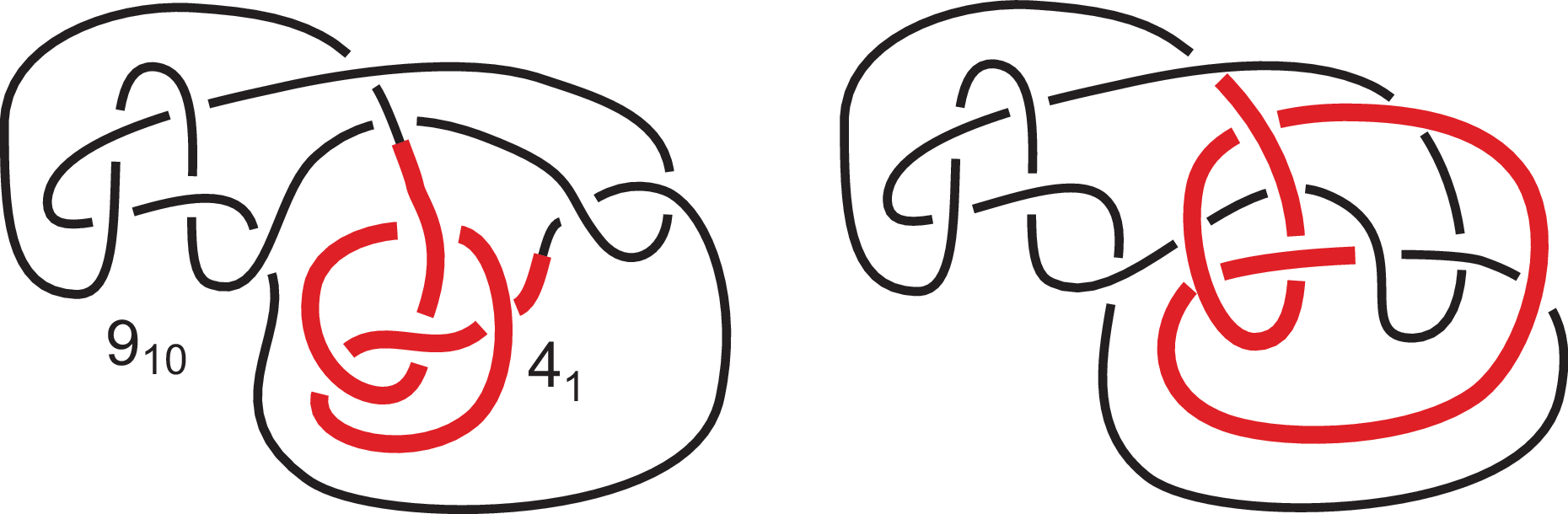}
\caption{Initial diagram for $4_1\#9_{10}$}\label{fig:connsum41b}
\end{center}
\end{figure}

\noindent Changing one crossing, underlined in the DT code (*), yields the diagram
with DT code 

\begin{center}
$[6, -10, 24, -20, -4, -22, -8, 26, 28, 30, -12, \underline{-2}, 14, 18, 16]$,\hskip.2in (**)
\end{center}

\noindent which SnapPy identifes as the knot $K15n4866$. This knot, from the authors'
work on the Bernhard-Jablan Conjecture \cite{bh21}, has unknotting number at most $2$.
There is, in fact, a crossing change in this diagram, underlined in the DT code (**),
to a knot with unknotting number one; the DT code

\begin{center}
$[6, -10, 24, -20, -4, -22, -8, 26, 28, 30, -12, 2, 14, 18, 16]$,
\end{center}

\noindent represents the knot $8_{14}$. The corresponding crossing changes in the 
knot diagram on the right in Figure \ref{fig:connsum41b} are 
shown in Figure \ref{fig:sum_to_814}.

\begin{figure}[h]
\begin{center}
\includegraphics[width=2.5in]{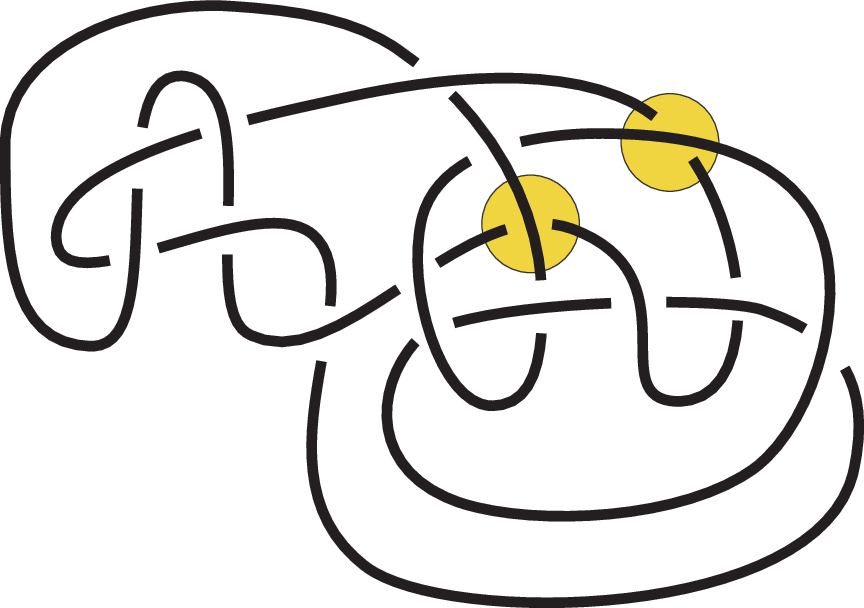}
\caption{Crossing changes to obtain the knot $8_{14}$}\label{fig:sum_to_814}
\end{center}
\end{figure}

The knot $8_{14}$ has unknotting number one; a minimal crossing diagram 
for the knot has DT code

\begin{center}
$[-10, -6, -14, -12, -16, -8, \underline{-2}, -4]$,\hskip.2in (***)
\end{center}

\noindent and a single crossing change, underlined in the DT code (***), yields the DT code

\begin{center}
$[-10, -6, -14, -12, -16, -8, 2, -4]$,
\end{center}

\noindent which is a diagram of the unknot. Consequently, $K_4=4_1\#9_{10}$ can be 
turned into the $8_{14}$ via the two crossing changes marked in 
Figure \ref{fig:sum_to_814},
which, after isotopy, can be transformed to the unknot by one further
crossing change. Therefore, $u(4_1\#9_{10})\leq 3$. 
On the other hand, $u(4_1)=1$, and the unknotting number $u(9_{10})=3$ was established by 
Owens \cite{ow08} using Heegaard Floer homology techniques. 
So $u(4_1)+u(9_{10})=1+3=4>3\geq u(4_1\#9_{10})$, and so 
the knot $9_{10}$ is a symbiont for the knot $4_1$. This establishes the
failure of additivity of the unknotting number for the pair $\{4_1,9_{10}\}$ ,
completing the proof of Theorem \ref{thm:four-one}.
\end{proof}

\smallskip

We note that Theorem \ref{thm:four-one} can be proved using three crossing changes in a 
single, relatively small, diagram for the knot $4_1\#9_{10}$. We obtain this
diagram by tracking the unknotting crossing for $8_{14}$ from its standard
8-crossing diagram through an isotopy to the 15-crossing diagram in Figure
\ref{fig:sum_to_814}. After some 
simplifying isotopies of the resulting diagram, we found the 
diagram for $4_1\#9_{10}$ given in Figure \ref{fig:all-in-one-41}. This has
35 crossings, and changing the indicated three crossings yields a 
diagram for the unknot.

\begin{figure}[h]
\begin{center}
\includegraphics[width=2.5in]{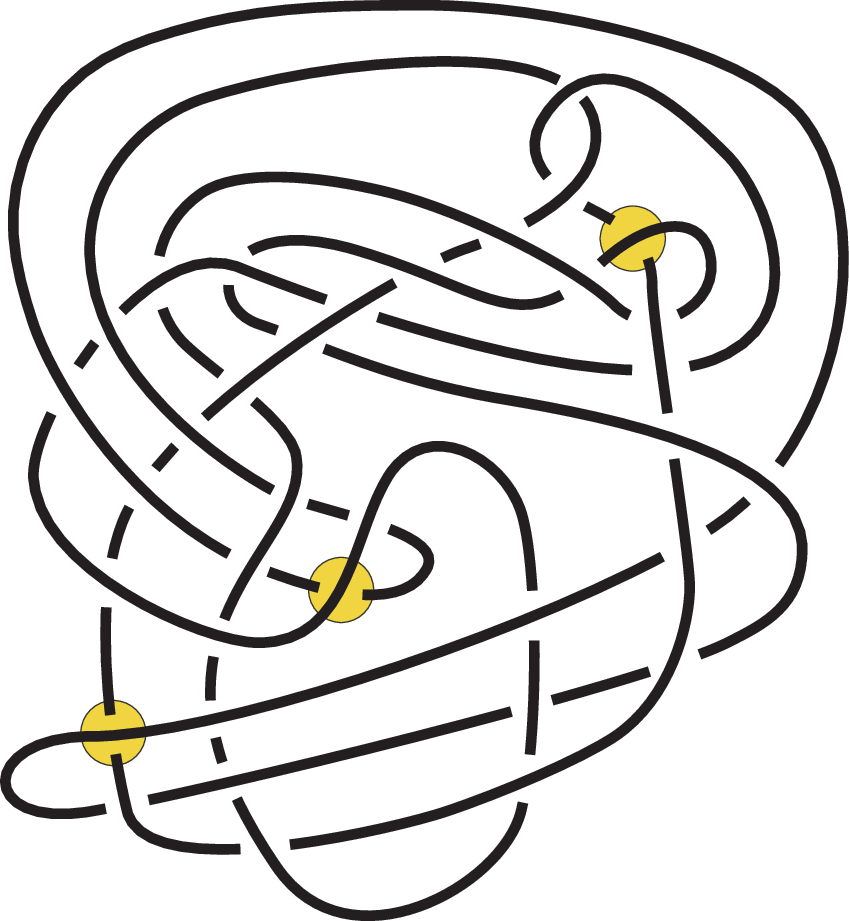}
\caption{From $4_1\#9_{10}$ to the unknot}\label{fig:all-in-one-41}
\end{center}
\end{figure}

\smallskip

\begin{proof}[Proof of Theorem \ref{thm:five-one}] 
Starting with the connected
sum $K_5$ of the knots $5_1$ and $8_2$, the knot $K_5$ has a 15-crossing 
diagram with a DT code given by

\begin{center}
$[6, \underline{12}, 24, -14, -16, -18, -2, -8, -10, 26, 28, 30, 4, 20, 22]$,\hskip.2in ($\dagger$)
\end{center}

\noindent as identified by SnapPy. In Figure \ref{fig:connsum51} we demonstrate an
isotopy from the standard sum diagram for $5_1\#8_2$ to the diagram given by the 
DT code in ($\dagger$).

\begin{figure}[h]
\begin{center}
\includegraphics[width=5in]{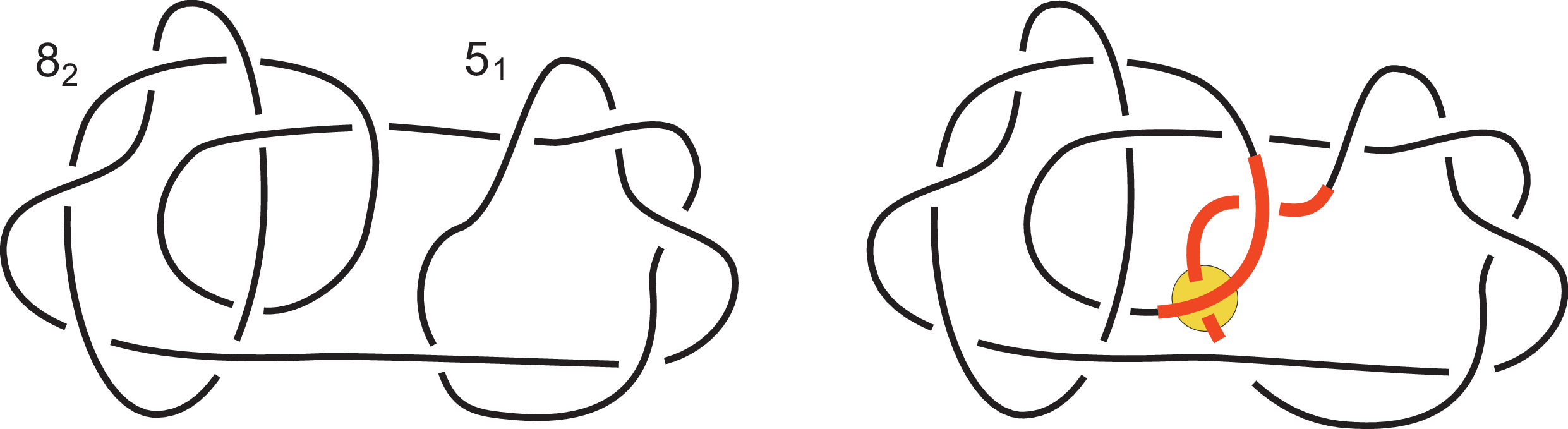}
\caption{Initial diagram for $5_1\#8_2$}\label{fig:connsum51}
\end{center}
\end{figure}

A single crossing change
in this diagram, underlined in the DT code ($\dagger$), yields a knot with DT code 

\begin{center}
$[6, -12, 24, -14, -16, -18, -2, -8, -10, 26, 28, 30, 4, 20, 22]$\hskip.2in ($\dagger\dagger$)
\end{center}

\noindent which SnapPy identifies as the knot $K15n72383$; 
the crossing change is indicated in Figure \ref{fig:connsum51}. This knot, from the 
authors' work on the Bernhard-Jablan Conjecture \cite{bh21}, has unknotting
number at most $2$. There is no crossing change in this diagram which will
lower the unknotting number; but $K15n72383$
has 88 distinct 15-crossing diagrams, and one of these has the DT code

\begin{center}
$[\underline{6}, 12, -16, 24, 26, 28, 20, 22, -30, -4, 14, 2, 8, 10, -18]$.\hskip.2in ($\dagger\dagger\dagger$)
\end{center}

\noindent SnapPy can verify that this and the previous DT code represent the same
knot. Changing a crossing in this diagram, underlined in the DT code ($\dagger\dagger\dagger$), 
yields the knot with DT code

\begin{center}
$[-6, 12, -16, 24, 26, 28, 20, 22, -30, -4, 14, 2, 8, 10, -18]]$
\end{center}

\noindent which SnapPy identifies as representing the knot $10_{129}$. This knot has
unknotting number one; from one of its DT codes,

\begin{center}
$[14, 8, 18, \underline{12}, -16, 4, 2, 20, -10, 6]$\hskip.2in ($\dagger \dagger \dagger \dagger$)
\end{center}

\noindent a single crossing change, underlined in the DT code ($\dagger \dagger \dagger \dagger$), 
yields the knot with the DT code

\begin{center}
$[14, 8, 18, -12, -16, 4, 2, 20, -10, 6]$
\end{center}

\noindent which can be isotoped to a knot with zero crossings, i.e., the unknot.
Consequently, $K_5=5_1\#8_2$ can be turned into the unknot by a 
crossing change, isotopy, crossing change, isotopy, and crossing change,
and so has unknotting number at most $3$. The knots $5_1$ and $8_2$ both have
signature $\pm 4$, depending on mirroring, and so both have
unknotting number at least $2$ \cite[Theorem~10.1]{mu65}. 
Any two crossing changes in the minimal crossing diagram for $5_1$ 
will result in the unknot, and there are five sets of $2$ crossing changes in
the minimal crossing diagram for $8_2$ that will result in the unknot. 
So $u(5_1)=u(8_2)=2$, and so 
$u(5_1)+u(8_2)=2+2=4>3\geq u(5_1\#8_2)$. Thus the knot $8_2$ is a symbiont 
for the knot $5_1$, establishing the
failure of additivity of unknotting number under connected sum for this pair. This completes
the proof of Theorem \ref{thm:five-one}. 
\end{proof} 


\smallskip

As mentioned in Section \ref{sec:intro}, our final example, the trefoil knot $3_1$,
may also have a symbiont as well, in the knot $10_6$. 
We show here that the knot $K_3=3_1\#10_6$, with summands
mirrored so that the knot has signature $\sigma(K_3)=\pm 2$, has unknotting 
number at most $3$. After our demonstration of this, we will discuss the status of
$u(10_6)$ in Section \ref{sec:ten-six}.

\smallskip

\begin{figure}[h]
\begin{center}
\includegraphics[width=5in]{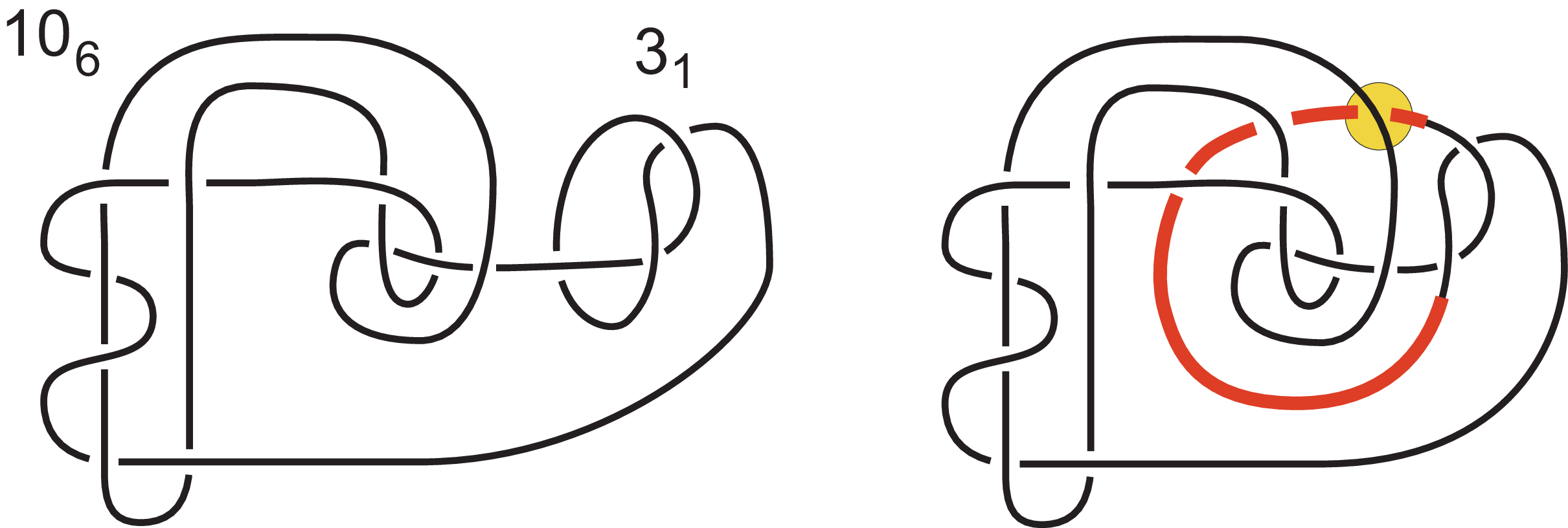}
\caption{Initial diagram for $3_1\#10_6$}\label{fig:connsum31}
\end{center}
\end{figure}

\begin{proof}[Proof of Theorem \ref{thm:three-one}]  Starting with the connected
sum $K_3$ of the knots $3_1$ and $10_6$, the knot $K_3$ has a 15-crossing 
diagram with a DT code given by

\begin{center}
$[4, 14, 20, -24, -26, -28, 16, 2, -22, \underline{-12}, -30, -18, -8, -10, -6]$\hskip.2in ($\ddagger$)
\end{center}

\noindent as identified by SnapPy \cite{snappy}. In Figure \ref{fig:connsum31} 
we show an isotopy from the standard connected sum diagram
for $3_1\#10_6$ to the diagram with DT code ($\ddagger$). A single crossing change
in this diagram, underlined in the DT code ($\ddagger$) and shown in a circle in Figure \ref{fig:connsum31})
yields a knot with DT code 

\begin{center}
$[4, 14, 20, -24, -26, -28, 16, 2, -22, 12, -30, -18, -8, -10, -6]$\hskip.2in ($\ddagger\ddagger$)
\end{center}

\noindent which SnapPy identifies as the knot $K15n9318$. This knot, from the 
authors' work on the Bernhard-Jablan Conjecture \cite{bh21}, has unknotting
number at most $2$. There is no crossing change in this diagram which will
lower the unknotting number; but the knot $K15n9318$
has 221 distinct 15-crossing diagrams, and one of these has the DT code

\begin{center}
$[4, 14, 20, 22, 24, \underline{16}, -26, 2, -28, 10, 6, 8, -30, -12, -18]$.\hskip.2in ($\ddagger\ddagger\ddagger$)
\end{center}

\noindent SnapPy can be used to verify that this code and the DT code ($\ddagger\ddagger$) represent the same
knot. Changing a crossing in this diagram, underlined in the DT code ($\ddagger\ddagger\ddagger$) yields the knot with DT code

\begin{center}
$[4, 14, 20, 22, 24, -16, -26, 2, -28, 10, 6, 8, -30, -12, -18]$
\end{center}

\noindent which SnapPy identifies as representing the knot $K12a1135$. This knot has
unknotting number one; from one of its DT codes,

\begin{center}
$[-16, -14, \underline{-24}, -18, -22, -20, -2, -4, -6, -8, -10, -12]$ ($\ddagger \ddagger \ddagger \ddagger$)
\end{center}

 \noindent a single crossing change, underlined in the DT code ($\ddagger \ddagger \ddagger \ddagger$), yields the knot with the DT code

\begin{center}
$[-16, -14, 24, -18, -22, -20, -2, -4, -6, -8, -10, -12]$
\end{center}

\noindent which can be simplified to a knot with zero crossings, i.e., the unknot.
Consequently, by reversing this process, we can go from the unknot to 
$K12a1135$ to $K15n9318$ by two crossing changes, with isotopies,
and, after isotopy, to $3_1\#10_6$ by a third crossing change. 
So $u(3_1\#10_6)\leq 3$.

The knot $3_1$ has unknotting number 1; any crossing change in its minimal
crossing diagram results in the unknot. The knot $10_6$ has 
$|\sigma(10_6)|=4$ so $u(10_6)\geq 2$, while for its minimal crossing projection
there are 35 sets of three crossings such that changing the crossings results in 
a diagram of the unknot;
so $u(10_6)\leq 3$. This gives $u(10_6)\in\{2,3\}$,
completing the proof of Theorem \ref{thm:three-one}.
\end{proof} 


\section{The unknotting number of $10_6$ and its consequences}\label{sec:ten-six}


As noted in Section~\ref{sec:intro}, the unknotting number of the knot $10_6$ has 
for many years had a reported value 
of $3$. This value can be found
in several sources, such as in the
table of Kawauchi \cite{ka96}. However, a list of corrections
to Kawauchi's table \cite{kawcorr} changes this to 2 \underline{or} 3.

\smallskip

If it were true that $u(10_6)=3$, then from Theorem~\ref{thm:three-one}
we have $u(3_1\#10_6)\leq 3<4=u(3_1)+u(10_6)$,
and so the knot $10_6$ would be a symbiont for the knot $3_1$ (and vice versa), 
establishing the
failure of additivity for the unknotting number of $3_1\#10_6$. 

On the other hand, if this is not true, and $u(10_6)=2$, then $K=10_6$ 
would be a counterexample
to the Bernhard-Jablan Unknotting Conjecture \cite{bernhard},\cite{jablan98}
with the smallest possible crossing number.
In particular, for the knot $10_6$,
from its unique (up to flypes) 10-crossing projection, the knots
obtained by changing a single crossing are $5_1$, $8_2$, and $8_6$,
all of which have unknotting number 2. The knots $5_1$ and $8_2$ 
were discussed above in the proof of Theorem \ref{thm:five-one},
while $u(8_6)=2$ was established in \cite{kamu86}.
The knot $10_6$ therefore has Bernhard-Jablan unknotting number $u_{BJ}(K)=3$. Consequently, 
if $u(10_6)=2$ we would have $u(10_6)=2<3=u_{BJ}(10_6)$. 

Which of these two situations holds is an interesting open problem.


\section{Further examples, via Gordian Adjacency}\label{sec:further_examples}


We have shown that the knot $4_1$ (and possibly $3_1$), with unknotting number one, 
and the knots $5_1$ and $8_2$, with unknotting number two, possess
symbiont knots $K$ so that connected sum with $K$ has unknotting
number lower than what is predicted by the Unknotting Additivity Conjecture.
These knots can be used to find still more knots which possess
symbionts, pairing them with (the same) knots $K$, to build
connected sums which fail the Unknotting Additivity Conjecture. 

The basic idea, as explored for the
knot $7_1$ in \cite{bh25}, is the following.
Suppose that $K_1$ is a knot with $u(K_1)=n$, and $K_2$ is a symbiont for $K_1$. 
Suppose also that $K$ is a knot with
$u(K)=n+m$, and $m$ crossing changes in a diagram for $K$ yields a diagram
for $K_1$. This is precisely the assertion that $K_1$ is Gordian adjacent to $K$.
Then $K\#K_2$ can be unknotted by making $m$ crossing changes
in a diagram for $K\#K_2$, resulting in a diagram for $K_1\#K_2$, 
and then we can make fewer than $n+u(K_2)=u(K_1)+u(K_2)$ crossing
changes to $K_1\#K_2$ to reach the unknot. So $u(K\#K_2)$ can be unknotted
with fewer than $m+(n+u(K_2))=u(K)+u(K_2)$ crossing changes, 
and so $K_2$ is also a symbiont for $K$. This gives:

\begin{lemma}\label{lemm:gordian} If $K_2$ is a symbiont for the knot
$K_1$ and $K_1$ is Gordian adjacent to $K$, then $K_2$ is a 
symbiont for $K$.
\end{lemma}

We can therefore find new knots that have symbionts by searching for 
knots which ($3_1$, provisionally, and) $4_1$, $5_1$, and $8_2$ 
are Gordian adjacent to. 

We can do this, experimentally, by searching for knots $L$
with known unknotting number $u(L)=2$ and for which one 
crossing change in a diagram for 
$L$ yields a diagram for ($3_1$ or) $4_1$. We have carried out such a 
search, for the 131 prime knots with 10 or fewer crossings which 
Knotinfo \cite{knotinfo} indicates have unknotting numbers equal to 
2. Using minimal crossing diagrams for these knots, obtained
using SnapPy \cite{snappy}, and changing crossings, we found that
24 of them are crossing adjacent to $4_1$. 
In addition, 50 of the knots are crossing adjacent to $3_1$; 
the two lists overlap in $10$ knots. In particular, 
the prime knots with crossing number at most 10 and $u(K)=2$ which we verified contain 
$4_1$ in a minimal unknotting sequence include the 24 knots in the set

\begin{center}
$S_2=\{8_{4},8_{12},8_{16},9_{8},9_{15},9_{17},9_{32},9_{40},10_{19},10_{35},
10_{36},10_{41},10_{45},10_{83}$,

$10_{89},10_{94},10_{105},10_{115},10_{121},10_{140},
10_{144},10_{155},10_{158},10_{163}\}$.
\end{center}

\noindent [Here we are using the numbering from Knotinfo for $10_{163}$; 
since SnapPy's knot database includes the Perko pair, it reports this knot as $10_{164}$.]
These knots $K$ therefore all share with $4_1$ the symbiont $9_{10}$, and so 
can be paired with a
this knot to give $u(K\#9_{10})<u(K)+u(9_{10})$. This inequality is, again, true for
any choice of mirroring, because this holds for $4_1\#9_{10}$; that is, 
both $9_{10}$ and $\overline{9_{10}}$ are symbionts for the knots in $S_2$.

Since both $5_1$ and $8_2$ are symbionts for one another, all of the knots in
the larger set $S_2^+=S_2\cup\{5_1,8_2\}$ are knots with unknotting number $2$
and possess symbionts. So the prime knots with crossing number at most 10 and unknotting
number 2 which possess symbionts include the knots in the set

\begin{center}
$S_2^+=\{5_{1},8_{2},8_{4},8_{12},8_{16},9_{8},9_{15},9_{17},9_{32},9_{40},10_{19},10_{35},
10_{36},10_{41}$,

$10_{45},10_{83},10_{89},10_{94},10_{105},10_{115},10_{121},10_{140},
10_{144},10_{155},10_{158},10_{163}\}$.
\end{center}

Carrying this approach a step further, we can search for prime knots
with $u(K)=3$ which contain one of the knots in
$S_2^+$ in a minimal unknotting
sequence. There are 25 prime knots with at most 10 crossings and
with unknotting number (known to be) 3, and
the ones which are crossing adjacent, in a miminal crossing number 
diagram, to one of the knots in the set $S_2^+$
include the 13 knots in the set

\begin{center}
$S_3=\{7_1,8_{19},9_{13},9_{38},9_{49},10_{2},
10_{46},10_{53},10_{103},10_{120},10_{134},10_{154},10_{161}\}$.
\end{center}

We note that only one of these knots, $10_{103}$, appears in $S_3$ 
from our experiment
through adjacency with a knot in $S_2$, being adjacent to the knot $8_{16}$. 
All of the other knots in $S_3$ were found through crossing adjacency 
with $5_1$ or $8_2$.
Because $9_{10}$ has symbiont $4_1$, the set $S_3^+=S_3\cup\{9_{10}\}$
gives 14 knots with unknotting number 3 that possess symbionts.
So the prime knots with crossing number at most 10 and unknotting
number 3 which possess symbionts include the knots in the set

\begin{center}
$S_3^+=\{7_1,8_{19},9_{10},9_{13},9_{38},9_{49},10_{2},
10_{46},10_{53},10_{103},10_{120},10_{134}$,

$10_{154},10_{161}\}$.
\end{center}

Finally, the four prime knots with at most 10 crossings and with 
unknotting number 4, namely the knots $9_1$, $10_{124}$, $10_{139}$, and $10_{152}$, 
are each crossing adjacent to either $7_1$ or $8_{19}$, using
minimal crossing diagrams, so each has, by Lemma \ref{lemm:gordian},
the knot $8_2$ as a symbiont, since 
$7_1$ and $8_{19}$ are in $S_3^+$ through adjacency to $5_1\in S_2^+$.
So the set of prime knots with crossing number at most 10 and unknotting
number 4 which possess a symbiont consists of the set

\begin{center}
$S_4=\{9_1,10_{124},10_{139},10_{152}\}$ .
\end{center} 

We note that the knot $5_1$ is the $(2,5)$ torus knot, and $8_{19}$ is the $(3,4)$ torus 
knot. These knots were excluded in Corollary~2.1 of \cite{bh25} which listed the torus knots
known to possess symbionts. We can therefore improve on part of that result.

\begin{corollary}\label{cor:torus}
Every non-trivial torus knot, except possibly for the trefoil knot $T(2,3)$, 
possesses a symbiont.
\end{corollary}


\smallskip

We can apply the same analysis using the knot $3_1$.
The prime knots with crossing number at most 10 and $u(K)=2$ which 
we verified contain $3_1$ in a 
minimal unknotting sequence include the 50 knots in the set

\smallskip

\begin{center}
$S_2^\prime=\{5_{1},7_{4},7_{5},8_{2},8_{6},8_{8},8_{16},8_{18},9_{7},9_{8},
9_{20},9_{31},9_{32},9_{46},9_{47},9_{48},10_{5}$,

$10_{14},10_{19},10_{20},10_{25},10_{34},10_{36},10_{40},10_{43},10_{68},10_{69},
10_{74},10_{75},10_{83},10_{86}$,

$10_{92},10_{94},10_{106},10_{109},10_{111},10_{115},10_{116},10_{117},10_{121},10_{125},10_{127},10_{135}$,

$10_{145},10_{148},10_{149},10_{150},10_{151},10_{160},10_{163}\}$.
\end{center}

\smallskip

Note that we do not yet know if any of the knots in $S_2^\prime$ which are not 
in $S_2^+$ possess symbionts; 
all of them do, if $u(10_6)=3$, based on our Theorem \ref{thm:three-one}.
In the same way, those knots with at most 10 crossings and unknotting number 3 which 
are crossing adjacent, in a miminal crossing number 
diagram, to one of the knots in $S_2^\prime$ include the 22 knots
in the set

\begin{center}
$S_3^\prime=\{7_{1},8_{19},9_{6},9_{9},9_{10},
9_{13},9_{16},9_{35},9_{38},9_{49},
10_{2},10_{46},10_{53}$,

$10_{66},10_{80},10_{101},10_{103},10_{120},10_{128},10_{134},10_{154},10_{161}\}$,
\end{center}

\noindent a set which contains $S_3^+$.
The only prime knots with 10 or fewer crossings and
unknotting number 3 which are not in the set $S_3^\prime$ 
are the knots $9_3$, $10_{49}$, and $10_{142}$ .
We do not, as of this writing, know whether or not any of these
three knots,
or any of the knots in $S_3^\prime\setminus S_3^+$, can be shown
to be crossing adjacent to $4_1$, and so belong in $S_3^+$, using 
crossing changes in non-minimal diagrams.


\smallskip

Consequently, of the $131+25+4=160$ prime knots through 10 
crossings with known unknotting number greater than 1, at 
least $26+13+4=43$ have a symbiont $K$ so that connected sum 
with $K$ fails the Unknotting Additivity
Conjecture. If the unknotting number of $10_6$ is in fact $3$, 
this total rises to $64+22+4=90$ knots.
It would be very interesting to know which of the remaining
knots possess symbionts, although for the knots with
unknotting number 1, this will require finding an \emph{ab initio}
construction of an unexpectedly efficient unknotting sequence.
The same may also be true for many of the remaining knots with 
unknotting number 2, as well.


\section{The road forward}\label{sec:road}


As with the authors' earlier work \cite{bh25}, establishing that 
$u(7_1\#\overline{7_1})\leq 5$, the results of this paper
relied heavily on the data generated for the authors' 
work on the the Bernhard-Jablan Conjecture, which included finding
upper bounds (in the form of BJ-unknotting numbers $u_{BJ}(K)$)
for all of the prime knots through 15 crossings .
This enabled us to identify the 
middle steps in unexpectedly short unknotting sequences
for $4_1\#9_{10}$, $5_1\#8_2$, and (potentially) $3_1\#10_6$.
It is the authors' hope that, with these added examples of the
failure of additivity for unknotting number, we may be able to 
start to get a glimpse of the mechanism by which failure
is occuring. It may be too much to hope that we could succeed
in replicating this phenomenon for any non-trivial knot, but
we may, with enough examples, be able to identify enough 
structure to find further examples by using theory alone, 
without computation.

One observation worth note is that, while for the original
example, $7_1\#\overline{7_1}$, the authors in their testing 
found only one initial diagram which could be used
to construct the unexpectedly short unknotting sequence, 
for each of the connected sums described in this paper, 
we have rediscovered dozens of short unknotting sequences for them
in our computations thus far. However, for each
of the knots $4_1$, $5_1$ and (potentially) $3_1$, 
\underline{only} the knots $9_{10}$, $8_2$, and (potentially!)
$10_6$, respectively, have occured as symbionts in those examples, despite
the fact that we have extensively tested 
randomly generated diagrams of all other knots, through 10
crossings, with unknotting numbers 3 (for $4_1$ and $3_1$)
or 2 (for $5_1$), as potential symbionts for these knots. 
That is, we have 
not been able to discover a different symbiont for any of 
$4_1$, $5_1$ or (potentially!) $3_1$ among these collections
of knots. Any knot that $8_2$ is 
Gordian adjacent to is also a symbiont for the knot $5_1$, 
and the same is true with the roles reversed, but none of the 
other 130 knots with crossing number up to 10 and unknotting number 2 have 
managed to appear in any of our searches for symbionts. We find this
circumstance rather intriguing. 
It suggests that among all of these potential candidates
for symbionts, only $9_{10}$ `fits' with $4_1$, and 
only $8_2$ fits with $5_1$. This in turn could perhaps provide clues for 
understanding how to match knots with symbionts more 
generally, by exploring what might distinguish these knots
from all of the others. 

With the three examples now in our possession, 
$4_1\#9_{10}$, $5_1\#8_2$, and $7_1\#\overline{7_1}$,
together with Gordian adjacency,
we can show that just over 17.6 percent of the 
prime knots through 10 crossings possess symbionts.
If $u(10_6)=3$ can be established, then from
Theorem \ref{thm:three-one} and the discussion
in Section \ref{sec:further_examples} this total would rise to 
36.9 percent.
We suggest that, in fact, asymptotically, \underline{most} knots 
will turn out to have symbionts; as crossing number grows, most knots,
we suspect, will contain 
either $3_1$ or $4_1$ in some minimal unknotting sequence. 
It is known that every diagram of a non-trivial knot admits 
crossing changes to the trefoil knot $3_1$ \cite{tan89}, and the 
same is nearly true for the figure-eight knot $4_1$
(one must exclude connected sums of $(2,n_i)$ torus knots \cite{tan89}). Although this
does not mean that $3_1$ or $4_1$ occur in a minimal unknotting sequence, it does
suggest that they might often be found there.

In light of the fact that two, and potentially all three, of
the first three knots in the standard knot table,
lie in the class of knots possessing symbionts,
we pose the following `anti-conjecture' to the Unknotting Additivity Conjecture:

\begin{conjecture}\label{conj:no-additivity} 
For every non-trivial knot $K$, there is a symbiont
knot $K^\prime$ for which $u(K\#K^\prime)<u(K)+u(K^\prime)$. That is, 
every non-trivial knot $K$ has a connected sum
for which the Unknotting Additivity Conjecture is false.
\end{conjecture}

\smallskip

We have focused here on the search for symbionts - knot pairs $K,K^\prime$
such that $K\#K^\prime$ fails additivity of unknotting number. With
the discovery of more symbiont pairs, it becomes even more interesting
to determine when unknotting additivity succeeds. There are several general results
which establish additivity, the first of which is the observation that
if $|\sigma(K)|=2u(K)$, $|\sigma(K^\prime)|=2u(K^\prime)$, and
$\sigma(K)\sigma(K^\prime)>0$, then $u(K\#K^\prime)=u(K)+u(K^\prime)$.
This follows directly from the additivity of signature \cite[Theorem~10.1]{mu65}. There is
also Scharlemann's result that $u(K\#K^\prime)=u(K)+u(K^\prime)$ when 
$u(K)=u(K^\prime)=1$ \cite{scha85}.
More specialized results include 
$u(\#_nK)=n$ if $u(K)=1$ and $\Delta_K(t)\neq 1$ \cite{ya08}, 
proved using the commutator subgroup of the knot group
and an associated lower bound on unknotting number \cite{mq06}, and 
$u((\#_n3_1)\#(\#_m\overline{3_1}))=n+m$ \cite{st02}, proved using 
3-move equivalence. There are also more isolated examples, such as
$u(4_1\#5_1)=u(4_1)+u(5_1)=3$ \cite{st04}, proved 
using the Brandt-Lickorish-Millett-Ho polynomial
\cite{blm86},\cite{ho85}, and
$u(5_1\#10_{132})=u(5_1)+u(10_{132})=3$ \cite{li20}, found
using the 4-dimensional clasp number. Other similar results
can be found throughout the literature.
A more thorough study of the success
of additivity for unknotting number is certain to uncover more
such pairs, and could lead to more comprehensive results.


\section{Verification code}\label{sec:verify}


For the convenience of the reader, we include Python code, suitable for 
running in SnapPy, which carries out the identifications in each of 
the unknotting sequences introduced in Section \ref{sec:counter}.
The code can be adapted to run in SageMath (after loading SnapPy with the
\emph{import snappy} command) by prepending
\emph{snappy.} to each instance of the words
\emph{Link} and \emph{Manifold}.

\smallskip

For $4_1\#9_{10}$ :

\begin{verbatim}
DTC1A=[6,-10,24,20,-4,-22,-8,26,28,30,-12,-2,14,18,16]
K1A=Link('DT:'+str(DTC1A))
K1A.simplify('global')
CC1A=K1A.deconnect_sum()
M1A1=CC1A[0].exterior()
M1A2=CC1A[1].exterior()
print(str(M1A1.identify())+'   '+str(M1A2.identify()))

DTC1B=DTC1A[:]
DTC1B[3]=-DTC1B[3]
M1B=Manifold('DT:'+str(DTC1B))
print(str(M1B.identify()))

DTC1C=DTC1B[:]
DTC1C[11]=-DTC1C[11]
M1C=Manifold('DT:'+str(DTC1C))
print(str(M1C.identify()))


DTC1D=[-10,-6,-14,-12,-16,-8,-2,-4]
M1D=Manifold('DT:'+str(DTC1D))
print(str(M1C.is_isometric_to(M1D)))

DTC1E=DTC1D[:]
DTC1E[6]=-DTC1E[6]
K1E=Link('DT:'+str(DTC1E))
K1E.simplify('global')
print(str(K1E))
\end{verbatim}

\smallskip

For $5_1\#8_2$ : 

\smallskip

\begin{verbatim}
DTC2A=[6,12,24,-14,-16,-18,-2,-8,-10,26,28,30,4,20,22]
K2A=Link('DT:'+str(DTC2A))
K2A.simplify('global')
CC2A=K2A.deconnect_sum()
M2A1=CC2A[0].exterior()
M2A2=CC2A[1].exterior()
print(str(M2A1.identify())+'   '+str(M2A2.identify()))
print(str(CC2A[0]))
print(str(M2A1.fundamental_group()))
print(str(CC2A[1]))
print(str(M2A2.fundamental_group()))
## a knot with (up to) 5 crossings and the stated knot group is 5_1

DTC2B=DTC2A[:]
DTC2B[1]=-DTC2B[1]
M2B=Manifold('DT:'+str(DTC2B))
print(str(M2B.identify()))

DTC2C=[6,12,-16,24,26,28,20,22,-30,-4,14,2,8,10,-18]
M2C=Manifold('DT:'+str(DTC2C))
print(str(M2B.is_isometric_to(M2C)))

DTC2D=DTC2C[:]
DTC2D[0]=-DTC2D[0]
M2D=Manifold('DT:'+str(DTC2D))
print(str(M2D.identify()))


DTC2E=[14,8,18,12,-16,4,2,20,-10,6]
M2E=Manifold('DT:'+str(DTC2E))
print(str(M2D.is_isometric_to(M2E)))

DTC2F=DTC2E[:]
DTC2F[3]=-DTC2F[3]
K2F=Link('DT:'+str(DTC2F))
K2F.simplify('global')
print(str(K2F))
\end{verbatim}

\smallskip

For $3_1\#10_6$ : 

\smallskip

\begin{verbatim}
DTC3A=[4,14,20,-24,-26,-28,16,2,-22,-12,-30,-18,-8,-10,-6]
K3A=Link('DT:'+str(DTC3A))
K3A.simplify('global')
CC3A=K3A.deconnect_sum()
M3A1=CC3A[0].exterior()
M3A2=CC3A[1].exterior()
print(str(M3A1.identify())+'   '+str(M3A2.identify()))

DTC3B=DTC3A[:]
DTC3B[9]=-DTC3B[9]
M3B=Manifold('DT:'+str(DTC3B))
print(str(M3B.identify()))

DTC3C=[4,14,20,22,24,16,-26,2,-28,10,6,8,-30,-12,-18]
M3C=Manifold('DT:'+str(DTC3C))
print(str(M3B.is_isometric_to(M3C)))

DTC3D=DTC3C[:]
DTC3D[5]=-DTC3D[5]
M3D=Manifold('DT:'+str(DTC3D))
print(str(M3D.identify()))


DTC3E=[-16,-14,-24,-18,-22,-20,-2,-4,-6,-8,-10,-12]
M3E=Manifold('DT:'+str(DTC3E))
print(str(M3D.is_isometric_to(M3E)))

DTC3F=DTC3E[:]
DTC3F[2]=-DTC3F[2]
K3F=Link('DT:'+str(DTC3F))
K3F.simplify('global')
print(str(K3F))
\end{verbatim}


\section{Acknowledgements}\label{sec:ackn}


The first author acknowledges support by a grant from
the Simons Foundation (Collaboration Grant number 525802).
The second author acknowledges support by a grant from
the Simons Foundation (Collaboration Grant number 581433).
The authors also acknowledge the support of
the Holland Computing Center at the University of 
Nebraska, which provided computing facilities on 
which the some of the work toward this project was carried out.




\end{document}